\documentclass[a4paper, 12pt]{amsart}

\usepackage{amssymb,amsmath}
\usepackage[dvips]{graphics}
\usepackage[dvips]{color}
\usepackage{graphicx}
\usepackage[english]{babel}
\usepackage[latin1]{inputenc}
\usepackage{comment}
\usepackage{todonotes}

\newtheorem{theorem}{Theorem}[section]
\newtheorem{lemma}[theorem]{Lemma}

\theoremstyle{definition}
\newtheorem{definition}[theorem]{Definition}

\newtheorem{remark}[theorem]{Remark}

\numberwithin{equation}{section}

\title[]
{A median approach to differentiation bases}
\author{Toni Heikkinen and Juha Kinnunen}
\address{Department of Mathematics, Aalto University, P.O. Box 11100, FI-00076 Aalto University, Finland}
\email{toni.heikkinen@aalto.fi,juha.k.kinnunen@aalto.fi}

\newcommand\rn{\mathbb R^n}
\newcommand\re{\mathbb R}
\newcommand\rv{\overline{\mathbb R}}
\newcommand\n{\mathbb N}
\newcommand\z{\mathbb Z}
\newcommand\q{\mathbb Q}
\newcommand\D{\mathbb D}

\newcommand\ph{\varphi}
\newcommand\eps{\varepsilon}
\newcommand\M{\operatorname{\mathcal M}}
\newcommand\cA{\mathcal A}
\newcommand\cB{\mathcal B}
\newcommand\cC{\mathcal C}
\newcommand\cD{\mathcal D}
\newcommand\cF{\mathcal F}

\newcommand\diam{\operatorname{diam}}

\providecommand{\ch}[1]{\text{\raise 2pt \hbox{$\chi$}\kern-0.2pt}_{#1}}

\providecommand{\vint}[1]{\mathchoice
          {\mathop{\vrule width 5pt height 3 pt depth -2.5pt
                  \kern -9.5pt \kern 1pt\intop}\nolimits_{\kern -5pt{#1}}}%
          {\mathop{\vrule width 5pt height 3 pt depth -2.6pt
                  \kern -6pt \intop}\nolimits_{\kern -3pt{#1}}}%
          {\mathop{\vrule width 5pt height 3 pt depth -2.6pt
                  \kern -6pt \intop}\nolimits_{\kern -3pt{#1}}}%
          {\mathop{\vrule width 5pt height 3 pt depth -2.6pt
                  \kern -6pt \intop}\nolimits_{\kern -3pt{#1}}}}
\begin{document}

\begin{abstract} 
We study a version of the Lebesgue differentiation theorem in which the integral averages are replaced with medians over Busemann--Feller differentiation bases. Our main result gives several characterizations for the differentiation property in terms of the corresponding median maximal function. As an application, we study pointwise behaviour in Besov and Triebel--Lizorkin spaces, where functions are not necessarily locally integrable. Most of our results apply also for functions defined on metric measure spaces.
%the interplay between generalized Lebesgue point properties and maximal function estimates  for $L^p$, Sobolev, Besov and Triebel--Lizorkin functions.
%Most of our results hold true for general differentiation bases in the setting of metric measure spaces. For some more subtle %results, we need the special covering properties of Euclidean spaces.
\end{abstract}


\date{\today}  

\maketitle
\begin{comment}
Kysymyksi\"a:
\begin{itemize}

\item Jos $C_p$ on Haj\l asz-kapasiteetti ja $p\le 1$, niin p\"ateek\"o
\[
C_p(\cup_{k=1}^\infty E_k)=\lim_{k\to \infty} C_p(E_k)
\]
kun $E_1\subset E_2\subset\dots$? Yleisemmin: p\"ateek\"o Lemma \ref{cap lemma 2} kun $p\le 1$ ja/tai $q\le 1$?

\item Todistetaanko lemmat \ref{cap lemma} ja \ref{cap lemma 2} vai todetaanko vain ett\"a todistukset samantapaisia kuin Sobo-tapauksessa.

\item Laitetaanko Hajlasz, Besov, TL m\"a\"aritelm\"at ja perustulokset kappaleeseen kaksi vai omaan kappaleeseen.

\end{itemize}

Introssa voisi mainita:
\begin{itemize}

\item Mediaanit ja niihin liittyv\"at maksimaalifunktiot 
ovat osoittautuneet hy\"odyllisiksi monissa yhteyksiss\"a, ks esim \cite{FZ}, \cite{Fu}, \cite{GKZ}, \cite{HIT}, \cite{HeKoTu}, \cite{HeTu2}, \cite{Hy}, \cite{Ha}, \cite{JPW}, \cite{JT}, \cite{J}, \cite{Kar}, \cite{L},  \cite{LP}, \cite{LP2}, \cite{PT}, \cite{St}, \cite{StTo}, \cite{Zh}. 

\item{} Pienill\"a $p$ Haj\l asz, Besov ja TL-funktiot eiv\"at v\"altt\"am\"att\"a integroituvia, jolloin keskiarvot hy\"odytt\"omi\"a (esim diskreetiss\"a konvoluutiossa ja kvasijatkuvan edustajan muodostamisessa)
Mediaanit sen sijaan toimivat aina, ks \cite{HIT}, \cite{HeKoTu}, \cite{HeTu2}.

\end{itemize}
\end{comment}

\section{Introduction}

If $f$ is a locally integrable function, then by the classical Lebesgue differentiation theorem, 
\begin{equation}\label{leb point}
\lim _{r\to0}\frac1{|B(x,r)|}\int_{B(x,r)}f(y)\,dy=f(x)
\end{equation}
for almost every $x\in\rn$. Here, $B(x,r)=\{y\in\rn: |x-y|<r\}$ is an open ball of radius $r>0$ centered at $x$ and $|A|$ denotes the Lebesgue measure of set $A\subset\rn$. If the integral averages in \eqref{leb point} are replaced by medians, or more generally, by $\gamma$-medians,
\[
m_f^\gamma(A)=\inf\big\{a\in\re: |\{x\in A: f(x)>a\}|< \gamma|A|\big\},
\]
where $A\subset\rn$ is a bounded and measurable set and $0<\gamma<1$, then for every measurable function $f:\rn\to[-\infty,\infty]$, with $|f(x)|<\infty$ for almost every $x\in\rn$, we have
\begin{equation}\label{gen leb point}
\lim _{r\to0}m^\gamma_f(B(x,r))=f(x)
\end{equation}
for almost every $x\in\rn$, see \cite{Fu}, \cite{PT}. 
Let $L^0(\rn)$ denote the set of all measurable functions $f\colon\rn\to[-\infty,\infty]$ such that $|f(x)|<\infty$ for almost every $x\in\rn$. Thus the Lebesgue differentiation theorem  with medians holds for functions in $L^0(\rn)$.

It is natural to ask whether \eqref{leb point} or \eqref{gen leb point} still hold true if the balls are replaced by
some other collections of sets converging to $x$.  In the case of integral averages, this problem has been studied by many authors, for example, see \cite{Guz1}, \cite{Guz2} and the references therein. It is known that if $\cB$ is a homothecy invariant Busemann--Feller basis (see Definition \ref{def:bases}), then
for every $f\in L^1(\rn)$, we have
\[
\lim_{\cB\,\ni B\to x}\frac{1}{|B|}\int_Bf(y)\,dy=f(x)
\]
for almost every $x\in\rn$ if and only if there exists a constant $C$ such that the corresponding maximal function 
\[
\M_\cB f(x)=\sup_{x\in B\in\cB}\frac{1}{|B|}\int_B|f(y)|\,dy
\] 
satisfies the weak type estimate
\[
|\{x\in\rn :\M_\cB f(x)>\lambda\}|\le C\lambda^{-1}\|f\|_{L^1(\rn)}
\]
for every $f\in L^1(\rn)$ and $\lambda>0$. 
A striking fact is that the qualitative Lebesgue differentiation theorem is characterized through a quantitative weak type estimate for the corresponding maximal function.

The purpose of this note is to obtain similar results for medians and the median maximal function
\[
\M^\gamma_\cB f(x)=\sup_{x\in B\in\cB}m^\gamma_{|f|}(B)
\] 
where $f:\rn\to[-\infty,\infty]$ is a measurable function with $|f(x)|<\infty$ for almost every $x\in\rn$.
Medians and related maximal functions have turned out to be useful in harmonic analysis and function spaces, see \cite{FZ}, \cite{Fu}, \cite{GKZ}, \cite{HIT}, \cite{HeKoTu}, \cite{HeTu2}, \cite{Hy}, \cite{Ha}, \cite{JPW}, \cite{JT}, \cite{J}, \cite{Kar}, \cite{L},  \cite{LP}, \cite{LP2}, \cite{PT}, \cite{St}, \cite{StTo}, \cite{Zh}. 
The main advantage of a median over an integral average is that it applies also when the function is not necessarily locally integrable.
This is relevant in certain function spaces, where functions are not necessarily locally integrable and thus integral averages are not defined. 
As we shall see, in many cases medians are more tractable than integral averages.

For homothecy invariant Busemann--Feller bases, we will prove the following theorem.

\begin{theorem}\label{thm1 in rn}
 %\cite{Guz1}, \cite{Guz2}
 Let $\cB$ be a homothecy invariant Busemann--Feller basis on $\rn$. Then the following claims are equivalent.
\begin{itemize}
\item[(1)] $\cB$ is a density basis, that is, for every measurable $A\subset \rn$,
\[
\lim_{\cB\,\ni B\to x}\frac{|A\cap B|}{|B|}=\chi_A(x)
\]
for almost every $x\in \rn$.

\item[(2)] For every  $0<\gamma<1$ and $f\in L^0(\rn)$,
\begin{equation}\label{gen leb point 2}
\lim_{\cB\,\ni B\to x}m^\gamma_f(B)=f(x)
\end{equation}
for almost every $x\in\rn$.

\item[(3)] For every $0<\gamma<1$, there exists a constant $C$ such that 
\[
|\{x\in\rn :\M^\gamma_\cB f(x)>\lambda\}|\le C|\{x\in\rn: |f(x)|>\lambda\}|
\]
for every $\lambda>0$ and $f\in L^0(\rn)$.

\item[(4)] For every $0<\gamma<1$ and $0<p<\infty$, there exists a constant $C$ such that 
\[
\|\M^\gamma_\cB f\|_{L^p(\rn)}\le C\|f\|_{L^p(\rn)}
\]
for every $f\in L^p(\rn)$.

\item[(5)] There exists $0<p<\infty$ such that, for every $0<\gamma<1$, $\lambda>0$ and for every sequence $(f_k)_{k\in\n}$ such that $\|f_k\|_{L^p(\rn)}\to 0$ as $k\to\infty$, we have 
\[
|\{x\in\rn:\M^\gamma_\cB f_k(B)>\lambda\}|\to 0
\]
as $k\to\infty$.

\item[(6)] There exists $0<p<\infty$ such that, for every $0<\gamma<1$ and $f\in L^p(\rn)$,
\[
\M^\gamma_\cB f(x)<\infty
\]
for almost every $x\in X$.
\end{itemize}
\end{theorem}

\begin{comment}
\begin{theorem}\label{thm1 in rn}
 %\cite{Guz1}, \cite{Guz2}
 Let $\cB$ be a homothecy invariant Busemann--Feller basis on $\rn$ and let $p>0$. Then the following claims are equivalent.
\begin{itemize}
\item[(1)] Basis $\cB$ is a density basis, that is, for every measurable $A\subset \rn$,
\[
\lim_{\cB\,\ni B\to x}\frac{|A\cap B|}{|B|}=\chi_A(x)
\]
for almost every $x\in \rn$.

\item[(2)] For every  $0<\gamma<1$ and $f\in L^p(\rn)$, with $1\le p<\infty$,
\[
\lim_{\cB\,\ni B\to x}m^\gamma_f(B)=f(x)
\]
for almost every $x\in\rn$.

\item[(3)] For every $0<\gamma<1$, there exists a constant $C$ such that 
\[
|\{x\in\rn :\M^\gamma_\cB f(x)>\lambda\}|\le C|\{x\in\rn: |f(x)|>\lambda\}|
\]
for every $\lambda>0$ and $f\in L^p(\rn)$, with $1<p<\infty$.

\item[(4)] For every $0<\gamma<1$, there exists a constant $C$ such that 
\[
\|\M^\gamma_\cB f\|_{L^p(\rn)}\le C\|f\|_{L^p(\rn)}
\]
for every  $f\in L^p(\rn)$, with $1<p<\infty$.

\item[(5)] For every $0<\gamma<1$, $\lambda>0$ and for every sequence $\{f_k\}_{k\in\n}$ such that $\|f_k\|_{L^p(\rn)}\to 0$ as $k\to\infty$, we have 
\[
|\{x\in\rn:\M^\gamma_\cB f_k(B)>\lambda\}|\to 0
\]
as $k\to\infty$.

\item[(6)] For every $0<\gamma<1$ and $f\in L^p(\rn)$,
\[
\M^\gamma_\cB f(x)<\infty
\]
for almost every $x\in X$.
\end{itemize}
\end{theorem}
\end{comment}

Homothecy invariance and special covering properties of the Euclidean spaces are needed only in showing that (1) implies (3). Assertions (1), (2), (5) and (6) are equivalent in more general metric measure spaces.

As an application, we study pointwise behaviour of functions in Besov and Triebel--Lizorkin spaces. We employ the definitions of Besov and Triebel--Lizorkin spaces in metric measure spaces introduced in \cite{KYZ}. This definition is motivated by the Haj\l asz--Sobolev spaces, see \cite{H} and for fractional scales \cite{Y}. The functions in these spaces are more regular than arbitrary measurable functions, but they are not necessarily locally integrable. Exceptional sets are measured with the corresponding capacity instead of the underlying measure. We give several characterizations of \eqref{gen leb point 2} in the context of metric measure spaces.
The results are new already in the Euclidean case with Lebesgue measure, but definitions of the function spaces in more general metric measure spaces give a transparent and flexible approach to pointwise behaviour.

%\section{Questions}

%\begin{itemize}

%\item Is there a differentiation basis $\cB$ and $u\in L^1$ such that 
%$\lim_{\cB\ni B\to x}m^\gamma_u(B)=u(x)$ for a.e.\,$x$ for some $\gamma$ but not for $\gamma'<\gamma$?

%\item Is there a differentiation basis $\cB$ and $u\in L^1$ such that there exists $E$ with $\mu(E)=0$ such that 
%$\lim_{\cB\ni B\to x}m^\gamma_u(B)=u(x)$ for every $0<\gamma<1$ and $x\in X\setminus E$, but
%\[
%\mu(\{x: \not\exists \lim_{\cB\ni B\to x} u_B \  \ \text{ or } \ \lim_{\cB\ni B\to x} u_B\not=u(x)\})
%\]
%has positive measure.

%(Since $|u|_B=\int_0^1m^\gamma_{|u|}(B)\,d\gamma$, Fatou's lemma implies that
%\[
%\liminf_{\cB\,\ni B\to 0}|u|_B=\liminf_{\cB\,\ni B\to 0}\int_0^1m^\gamma_{|u|}(B)\,d\gamma\ge |u(x)|
%\]
%for a.e. $x$.)
%Suppose that, for some $p>1$, $\Mu^p$ is finite almost everywhere. If $\lim_{\cB\ni B\to x}m^\gamma_u(B)=u(x)$ for every %$0<\gamma<1$ and $x\in X\setminus E$, then $\lim_{}u_B=u(x)$
%\end{itemize}

\section{Preliminaries}\label{sec: preliminaries}

\subsection{Basic assumptions}
In this paper,  $X=(X, d,\mu)$ denotes a metric measure space equipped with a metric $d$ and a Borel regular
outer measure $\mu$, for which the measure of every ball is positive and finite.A measure $\mu$ is \emph{doubling} if there exists a constant $C_d$, such that
\[
\mu(B(x,2r))\le C_d\mu(B(x,r))
\]
for every ball $B(x,r)=\{y\in X:d(y,x)<r\}$, where $x\in X$ and $r>0$.
The doubling condition is equivalent to existence of constants $C$ and $Q$ such that
\begin{equation}\label{doubling dim}
\frac{\mu(B(y,r))}{\mu(B(x,R))}\ge C\Big(\frac rR\Big)^Q 
\end{equation}
for every $0<r\le R$ and $y\in B(x,R)$. 

The integral average of a locally integrable function $f$ over a measurable set $A$ of positive and finite measure is denoted by
\[
f_A=\vint{A}f\,d\mu=\frac{1}{\mu(A)}\int_A f\,d\mu.
\]

The characteristic function of a set $E\subset X$ is denoted by $\ch{E}$. 
%Denote $\rv=[-\infty,\infty]$.
$L^0(X)$ is the set of all measurable functions $f\colon X\to[-\infty,\infty]$ such that $|f(x)|<\infty$ for almost every $x\in X$.
In general, $C$ denotes a positive and finite constant whose value are not necessarily same at each occurrence.
When we want to emphasize that the constant depends on parameters $a,b,\dots$, we write $C=C(a,b,\dots)$.

\subsection{$\gamma$-median} 
%Our important tools are median values and median maximal functions, which have been studied and used in different problems of analysis for example in  
%\cite{FZ}, \cite{Fu}, \cite{GKZ}, \cite{HIT}, \cite{JPW}, \cite{JT}, \cite{J}, \cite{L}, \cite{LP}, \cite{PT}, \cite{St} and \cite{Zh}. 
%In the theory of Besov and Triebel--Lizorkin spaces, they are extremely useful when $0<p\le1$ or $0<q\le 1$.

%\begin{definition}
%Let $0<\gamma\le 1/2$. The \emph{$\gamma$-median} of a measurable function
%$u\colon X\to\rv$ over a set $A$ of finite measure is
%\[
%m_u^\gamma(A)=\inf\big\{a\in\re: \mu(\{x\in A: u(x)>a\})< \gamma\mu(A)\big\}.
%\]
%\end{definition}
Let $A\subset X$ be a measurable set with $\mu(A)<\infty$. 
For $0<\gamma< 1$, the $\gamma$-median of a measurable function
$f\colon X\to[-\infty,\infty]$ over $A$ is
\[
m_f^\gamma(A)=\inf\big\{a\in\re:\mu(\{x\in A: f(x)>a\})< \gamma\mu(A)\big\}.
\]
Note that  $m_f^\gamma(A)$ is finite, if $f\in L^0(A)$ and $0<\mu(A)<\infty$. 
%The main advantage of the $\gamma$-median compared to the integral average is that it applies to functions that are not necessarily integrable over $A$. Parameter $\gamma$ does not play an essential role in what follows. However, we want to emphazise that our results do not only apply  to the standard convention $\gamma=1/2$, but also to a general $\gamma$ with $0<\gamma<1$.
We list some basic properties of the $\gamma$-median below. Properties (1), (2), (4), (6), (7) and (8) are proved in the Euclidean setting in \cite[Propositions 1.1 and 1.2]{PT}.
% and (i) from \cite[Theorem 2.1]{PT}. 
The remaining properties follow immediately from the definition.

\begin{lemma}\label{median lemma} 
Let $A,B\subset X$ be measurable sets with $\mu(A)<\infty$ and $\mu(B)<\infty$ and assume that $f,g:X\to[-\infty,\infty]$ are measurable functions. 
\begin{itemize}
\item[(1)] If $0<\gamma\le\gamma'<1$, then $m_{f}^{\gamma}(A)\ge m_{f}^{\gamma'}(A)$.
\item[(2)] If $f\le g$ almost everywhere in $A$, then  $m_{f}^{\gamma}(A)\le m_{g}^{\gamma}(A)$.
\item[(3)] If $A\subset B$ and $\mu(B)\le C\mu(A)$, then $m_{f}^{\gamma}(A)\le m_{f}^{\gamma/C}(B)$.
\item[(4)]  $m_f^\gamma(A)+c=m_{f+c}^\gamma(A)$  for every $c\in\re$.
\item[(5)] $m_{c\,f}^\gamma(A)=c\,m_{f}^\gamma(A)$ for every $c>0$.
\item[(6)] $|m_{f}^\gamma(A)|\le m_{|f|}^{\min\{\gamma,1-\gamma\}}(A)$. 
\item[(7)] $m_{f+g}^\gamma(A)\le m_{f}^{\gamma_1}(A)+m_{g}^{\gamma_2}(A)$ whenever $\gamma_1+\gamma_2=\gamma$.
\item[(8)] For every $p>0$, 
\[
m_{|f|}^\gamma(A)\le \Big(\gamma^{-1}\vint{A}|f|^p\,d\mu\Big)^{1/p}.
\]
\item[(9)] If $\cB$ is a differentiation basis and $f$ is continuous, then for every $x\in X$,
\[
\lim_{\cB\,\ni B\to x} m_{f}^\gamma(B)=u(x).
\]
%\item[(j)]  If $u\in L^0(X)$, then 
%there exists a set $E$ with $\mu(E)=0$ such that \[\lim_{r\to 0} m_{u}^\gamma(B(x,r))=u(x)\]
% for every $0<\gamma\le 1/2$ and $x\in X\setminus E$. 
\end{itemize}
\end{lemma}

Property (9) above asserts that the pointwise value of a continous function can be obtained as a limit of medians over small balls. In this sense medians behave like integral averages of continuous functions. %points. 

%\begin{definition}
%Let $u\in L^0(A)$. 
%A point $x$ is a \emph{generalized Lebesgue point} of $u$, if 
%\[
%\lim_{r\to 0} m_{u}^\gamma(B(x,r))=u(x)
%\]
%for all $0<\gamma\le1/2$.
%\end{definition}

\subsection{Differentiation bases}
A collection $\cB=\cup_{x\in X}\cB(x)$, where $\cB(x)$ consist of  bounded measurable sets containing $x$, is called a \emph{differentiation basis} if, for every $x\in\rn$ and $\eps>0$, there exists $B\in\cB(x)$ such that $\diam(B)<\eps$. 
Let $f\in L^0(X)$. For a diffentiation basis $\cB$, we denote
\[
\limsup_{\cB\,\ni B\to x} m^\gamma_f(B)=\lim_{\eps\to 0}\sup\{m^\gamma_u(B): B\in\cB(x), \, \diam(B)<\eps\},
\]
\[
\liminf_{\cB\,\ni B\to x} m^\gamma_f(B)=\lim_{\eps\to 0}\inf\{m^\gamma_u(B): B\in\cB(x), \, \diam(B)<\eps\}
\]
and 
\[
\lim_{\cB\,\ni B\to x} m^\gamma_f(B)=\limsup_{\cB\,\ni B\to x} m^\gamma_f(B),%=\liminf_{\cB\,\ni B\to x} m^\gamma_u(B)
\]
if the limes superior and the limes inferior coincide.

\begin{definition}\label{def:bases}
\begin{itemize}
\item[]
\item[(1)] A differentiation basis $\cB$ is a \emph{Busemann--Feller basis} if each $B\in \cB$ is open and $x\in B\in\cB$ implies that $B\in\cB(x)$.
\item[(2)] A differentiation basis $\cB$ is a \emph{density basis} if, for every measurable $A\subset X$,
\[
\lim_{\cB\,\ni B\to x}\frac{\mu(A\cap B)}{\mu(B)}=\chi_A(x)
\]
for almost every $x\in X$.
\item[(3)] A differentiation basis $\cB$ is \emph{homothecy invariant} if $B\in\cB$ implies that $B'\in \cB$ for every $B'$ homothetic to $B$.
\end{itemize}
\end{definition}
We begin with a version of the Lebesgue density theorem for $\gamma$-medians.

\begin{theorem}
 %\cite{Guz1}, \cite{Guz2}
 The following claims are equivalent.
\begin{itemize}
\item[(1)] $\cB$ is a density basis.
%, that is, for every measurable $A\subset X$,
%\[
%\lim_{\cB\,\ni B\to x}\frac{\mu(A\cap B)}{\mu(B)}=\chi_A(x)
%\]
%for a.e. $x\in X$.

\item[(2)] For every $f\in L^0(X)$, there exists a set $E$ with measure zero such that
\[
\lim_{\cB\,\ni B\to x}m^\gamma_{|f-f(x)|}(B)=0
\]
for every $0<\gamma<1$ and $x\in X\setminus E$.

\item[(3)] For every $f\in L^0(X)$, there exists a set $E$ with measure zero such that
\[
\lim_{\cB\,\ni B\to x}m^\gamma_f(B)=f(x)
\]
for every $0<\gamma<1$ and $x\in X\setminus E$.
\end{itemize}
\end{theorem}

\begin{proof} We begin with showing that (1) implies (2).
First we prove that for fixed $f\in L^0(X)$ and $0<\gamma<1$,
\[
\lim_{\cB\,\ni B\to x}m^\gamma_f(B)=f(x)
\]
for almost every $x\in X$. The proof is a slight modification of the proof of \cite[Theorem 3.1]{PT}. 
For $k=1,2,\dots$ and $j\in\z$, denote 
\[
E_{k,j}=\{x: (j-1)2^{-k}\le f(x)<j2^{-k}\}
\]
and 
\[
S_k=\sum_{j\in\z}(j-1)2^{-k}\chi_{E_{k,j}}.
\]
Then $\mu(X\setminus\cup_{k,j}E_{k,j})=0$. Since $S_k\le u\le S_k+2^{-k}$ almost everywhere, we have 
\begin{equation}\label{eq 1 for Sk}
m^\gamma_{S_k}(B)\le m^\gamma_f(B)\le m^\gamma_{S_k}(B)+2^{-k}
\end{equation}
for every $B\in\cB$. Denote 
\[
A_{k,j}=\left\{x\in E_{k,j}: \lim_{\cB\,\ni B\to x}\frac{\mu(E_{k,j}\cap B)}{\mu(B)}=1\right\}
\]
and $A=\cup_{k=1}^\infty\cup_{j\in\z}A_{k,j}$. Then, by assumption, $\mu(X\setminus A)=0$. We show that
\[
\lim_{\cB\,\ni B\to x}m^\gamma_f(B)=f(x)
\]
for every $x\in A$. Let $x\in A$ and $\eps>0$. Choose $k$ such that
$2^{-k+1}<\eps$. Then $x\in A_{k,j}$ for some $j$. Thus, for all $B\in\cB$ with $\diam(B)$ small enough,
\[
\frac{\mu(E_{k,j}\cap B)}{\mu(B)}>\max\{\gamma,1-\gamma\}.
\]
For such $B$, we have 
\begin{equation}\label{eq 2 for Sk}
m^\gamma_{S_k}(B)=(j-1)2^{-k}.
\end{equation}
Indeed, 
\[
\begin{split}
\mu(\{y\in B: S_k(y)>(j-1)2^{-k}\})
&\le \mu(B\setminus E_{k,j})\\
&=\mu(B)-\mu(B\cap E_{k,j})<\gamma\mu(B)
\end{split}
\]
and thus $m^\gamma_{S_k}(B)\le (j-1)2^{-k}$. On the other hand, for every $\delta>0$,
\[
\mu(\{y\in B: S_k(y)>(j-1)2^{-k}-\delta\})\ge \mu(B\cap E_{k,j})>\gamma\mu(B)
\]
which implies that $m^\gamma_{S_k}(B)\ge (j-1)2^{-k}$.
By combining \eqref{eq 1 for Sk} and \eqref{eq 2 for Sk}, we obtain
\[
|m^\gamma_f(B)-u(x)|\le |m^\gamma_f(B)-m^\gamma_{S_k}(B)|+|m^\gamma_{S_k}(B)-f(x)|\le 2^{-k+1}<\eps
\]
for every $B\in\cB$ with $\diam(B)$ small enough. 
%This concludes the proof that, for fixed $u\in L^0$ and $0<\gamma<1$,
%$\lim_{\cB\,\ni B\to x}m^\gamma_u(B)=u(x)$ for a.e. $x\in X$.

Then assume that $f\in L^0(X)$. By what we have shown above, there exists $E\subset X$ with $\mu(E)=0$ such that
\[
\lim_{\cB\,\ni B\to x}m^\eta_{|f-q|}(B)=|f(x)-q|
\]
whenever $\eta\in \q\cap [0,1]$, $q\in\q$ and $x\in X\setminus E$. 
If $0<\gamma<1$, $x\in X\setminus E$ and $\eps>0$,
we choose $q\in\q$ such that $|f(x)-q|<\eps$ and $\eta\in\q$ such that $0<\eta\le \gamma$. This implies
\[
\lim_{\cB\,\ni B\to x}m^\gamma_{|f-f(x)|}(B)
\le\lim_{\cB\,\ni B\to x}m^\eta_{|f-f(x)|}(B)
<\lim_{\cB\,\ni B\to x}m^\eta_{|f-q|}(B)+\eps<2\eps,
\]
which concludes the proof that (1) implies (2). 

Then we show that (2) implies (3).
Since
\[
|f(x)-m^\gamma_f(B)|=|m^\gamma_{(f-f(x))}(B)|
\le m^{\min\{\gamma,1-\gamma\}}_{|f-f(x)|}(B),
\]
the claim follows immediately.

Finally, we conclude that (3) implies (1). 
It is easy to see that  \[m^\gamma_{\chi_A}(B)=\chi_{(0,\mu(A\cap B)/\mu(B)]}(\gamma)\]
for every $0<\gamma <1$ and $A,B\subset X$. 
This completes the proof.
\end{proof}

\section{Proof of Theorem \ref{thm1 in rn}}

%\subsection{From Lebesgue point property to weak type estimate}
%Next, we prove that the generalized Lebesgue point property for $L^p$ implies a weak type estimate for the median maximal %function. 

Next, we prove that if $\cB$ is a homothecy invariant Busemann--Feller density basis on $\rn$, then $\M^\gamma_\cB$ is bounded on $L^p(\rn)$  for $p\ge 0$. The main ingredient of the proof is the following characterization of a density basis (\cite[Theorem 1.2]{Guz1}).

\begin{theorem}\label{Guz1 thm 1.2}
Let $\cB$ be a homothecy invariant Busemann--Feller basis on $\rn$. Then the following
are equivalent.
\begin{itemize}
\item[(1)] $\cB$ is a density basis. 
\item[(2)] For every $0<\gamma<1$, there exists a constant $C=C(\gamma)$ such that 
\begin{equation}\label{Guz thm 1.2 b}
|\{x\in\rn:\M_\cB\chi_A(x)>\gamma\}|\le C|A|
\end{equation}
for every bounded measurable set $A\subset\rn$.
\end{itemize}

\begin{remark}
If \eqref{Guz thm 1.2 b} holds for every bounded measurable set $A\subset\rn$, then
it holds for every measurable $A\subset\rn$. Indeed, if $A$ is measurable, then $A_k=A\cap B(0,k)$ is bounded and measurable for every $k$ and so
\[
\begin{split}
|\{x\in\rn:\M_\cB\chi_A(x)>\gamma\}|&=\lim_{k\to\infty}|\{x\in\rn:\M_\cB\chi_{A_k}(x)>\gamma\}|\\
 &\le C\lim_{k\to\infty}|A_k|
 =C|A|.
\end{split}
\]
\end{remark}

\end{theorem}

We also need the following simple lemma which follows easily from the definitions.

\begin{lemma}\label{Mgamma M lemma} 
For every $f\in L^0(\rn)$, $x\in \rn$, $0<\gamma<1$ and $\lambda>0$, we have
\[
\M^\gamma_\cB f(x)>\lambda\iff \M_\cB\chi_{\{|f|>\lambda\}}(x)>\gamma.
\]
\end{lemma}

\begin{theorem}  Let $\cB$ be a homothecy invariant Busemann--Feller basis on $\rn$. 
If $\cB$ is a density basis then, for every $0<\gamma<1$, there exists a constant $C=C(\gamma)$ such that 
\begin{equation}\label{L0 for Mgamma}
|\{x\in\rn:\M^\gamma_\cB f(x)>\lambda\}|\le C|\{x\in\rn:|f(x)|>\lambda\}|
\end{equation}
for every  $f\in L^0(\rn)$ and $\lambda>0$. Consequently, for every $p>0$, 
\begin{equation}\label{Lp for Mgamma}
\|\M^\gamma_\cB f\|_{L^p(\rn)}\le (Cp)^{\frac1p}\|f\|_{L^p(\rn)}
\end{equation}
for every $f\in L^p(\rn)$.
\end{theorem}

\begin{proof}
By Lemma \ref{Mgamma M lemma}, 
\[
|\{x\in\rn:\M^\gamma_\cB f(x)>\lambda\}|
=|\{x\in\rn:\M_\cB\chi_{\{|f|>\lambda\}}(x)>\gamma\}|
\] 
for every  $f\in L^0(\rn)$, $0<\gamma<1$ and $\lambda>0$. 
By Theorem \ref{Guz1 thm 1.2}, there exists a constant $C=C(\gamma)$
such that \eqref{L0 for Mgamma} holds for every $f\in L^0(\rn)$ and $\lambda>0$.
Since
\[
\|f\|_{L^p(\rn)}^p=p\int_0^\infty \lambda^{p-1}|\{x\in\rn: |f(x)|>\lambda\}|\,d\lambda,
\]
estimate \eqref{Lp for Mgamma} follows from \eqref{L0 for Mgamma}.
\end{proof}

\begin{remark}
Inequality \eqref{L0 for Mgamma} implies that $\M^\gamma_\cB$ is also bounded on other rearrangement invariant
spaces such as Orlicz and Lorentz spaces.
\end{remark}

Next, we prove the implication (6)$\implies$(5) of Theorem \ref{thm1 in rn}. We need a couple of simple lemmas.

\begin{lemma}\label{med lemma 1} Let $A\subset X$ be a measurable set with $0<\mu(A)<\infty$, $f\in L^0(A)$ and $0<\gamma<1$. Then 
\[
\lim_{0<\eps\to 0}m^{\gamma-\eps}_{f}
(A)=m^\gamma_f(A).
\]
\end{lemma}

\begin{proof}
Let $\lambda>m^\gamma_f(A)$. Then 
\[
\mu(\{x\in A: f(x)>\lambda\})<\gamma\mu(A).
\]
Thus for $\eps$ small enough, 
\[
\mu(\{x\in A: f(x)>\lambda\})<(\gamma-\eps)\mu(A),
\]
which implies that $m^{\gamma-\eps}_f(A)\le \lambda$. This implies that 
\[
\limsup_{0<\eps\to 0}m^{\gamma-\eps}_f(A)\le m^\gamma_f(A).
\]
Since $m^\gamma(A)\le m^{\gamma-\eps}_f(A)$ for every $0<\eps<\gamma$, the claim follows.
\end{proof}

\begin{lemma}\label{med lemma 2}  Let $A\subset X$ be a measurable set with $0<\mu(A)<\infty$ and $0<\gamma< 1$. 
If $f_i\in L^0(A)$, $i=1,2,\dots$, and $f_i\to f$ in $L^0(A)$ as $i\to\infty$,  then 
\[
\lim_{i\to \infty}m^{\gamma}_{f_i}
(A)=m^\gamma_f(A).
\]
\end{lemma}

\begin{proof} 
%KORJAA!!!!! If $u_i\to u$ in $L^1(A)$,  then $|\{x\in A: u_i(x)>\lambda\}|\to |\{x\in A: u(x)>\lambda\}|$.
Let $\lambda>m^\gamma_u(A)$. Then 
\[
\mu(\{x\in A: f(x)>\lambda\})<\gamma\mu(A).
\]
For $\eps>0$, we have 
\[
\begin{split}
\mu\{x\in A: f_i(x)>\lambda+\eps\})
&\le \mu(\{x\in A: f(x)>\lambda\})\\
&+\mu(\{x\in A: |f_i(x)-f(x)|>\eps\})
<\gamma \mu(A)
\end{split}
\]
for  $i$ large enough, which implies that 
\[
\limsup_{i\to\infty}m^{\gamma}_{f_i}(A)\le \lambda+\epsilon.
\]
Thus 
\[
\limsup_{i\to \infty}m^{\gamma}_{f_i}(A)\le m^\gamma_f(A).
\]

Then we show that
\[
m^\gamma_f(A)\le \liminf_{i\to \infty}m^{\gamma}_{f_i}(A).
\]
Let 
\[
\lambda>\liminf_{i\to \infty}m^{\gamma}_{f_i}(A).
\]
Then there are arbitrarily large $i$ such that
\[
\mu(\{x\in A: f_i(x)>\lambda\})<\gamma \mu(A),
\]
which implies that, for $\eps,\delta>0$, 
\[
\begin{split}
\mu(\{x\in A: f(x)>\lambda+\eps\})
&\le \liminf_{i\to\infty}\mu(\{x\in A: f_i(x)>\lambda\})\\
&+\lim_{i\to\infty}\mu(\{x\in A: |f(x)-f_i(x)|>\eps\})\\
&\le \gamma\mu(A)
<(\gamma-\delta)\mu(A).
\end{split}
\]
Thus, $m^{\gamma-\delta}_f(A)\le\lambda+\eps$. By Lemma \ref{med lemma 1}, $m^{\gamma}_f(A)\le\lambda+\eps$. The claim follows by passing $\eps\to 0$ and $\lambda\to \liminf_{i\to \infty}m^{\gamma}_{f_i}(A)$.
\end{proof}

\begin{lemma}\label{med lemma 3}
Let $A_i\subset X$, $i=1,2,\dots$, be measurable sets such that $A_i\subset A_{i+1}$ for every $i$ and $A=\cup_{i=1}^\infty A_i$
is of finite measure. Then, for every $0<\gamma<1$ and $f\in L^0(X)$, we have
\[
\lim_{i\to\infty}m^\gamma_f(A_i)= m^\gamma_f(A).
\]
\end{lemma}

\begin{proof}
It suffices to show that 
\[
\liminf_{i\to\infty}m^{\gamma}_f(A_i)\ge m^{\gamma}_f(A).
\]
Let $\eps,\delta>0$. By definition, 
\[
\mu(\{x\in A: f(x)>m^{\gamma-\delta}_u(A)-\eps\})\ge(\gamma-\delta)\mu(A).
\]
It follows that
there is $i_0=i_0(\delta)$ such that 
\[
\mu(\{x\in A_i: f(x)>m^{\gamma-\delta}_f(A)-\eps\})>\gamma\mu(A_i)
\]
when $i\ge i_0$.
Thus 
\[
\liminf_{i\to\infty}m^{\gamma}_f(A_i)\ge m^{\gamma-\delta}_f(A)-\eps.
\]
The claim follows by letting $\delta,\eps\to 0$.
\end{proof}

\begin{lemma}\label{countable}
Let $X$ be separable and let $\cB$  be a Busemann--Feller basis on $X$. Then there exists a countable Busemann--Feller basis
$\cB'$ such that 
\begin{equation}\label{B B'}
%\[
\M^\gamma_\cB f(x)\le \M^\gamma_{\cB'}f(x)\le \M^{\gamma/2}_\cB f(x)
%\]
\end{equation}
for every $0<\gamma<1$, $f\in L^0(X)$ and $x\in X$. 
\end{lemma}

\begin{proof}
Fix a countable dense set $A\subset X$. Then 
\[
\cC=\{\cup_{i=1}^k B(x_i,r_i): k\in\n,\, x_i\in A,\, r_i\in\q\}
\]
is countable. Since $B\in\cB$ is open, 
\[
B=\cup\{B(x,r): x\in B\cap A,\,r\in\q,\, B(x,r)\subset B\}.
\]
Hence, for every $B\in\cB$, there exists a nondecreasing sequence $(C_i)$, $C_i\in\cC$, such that $\cup_{i=1}^\infty C_i= B$.
It follows that $\lim_{i\to\infty}\mu(C_i)=\mu(B)$ and,
by Lemma \ref{med lemma 3}, 
\begin{equation}\label{lim}
\lim_{i\to\infty}m^\gamma_{|f|}(C_i)= m^\gamma_{|f|}(B)
\end{equation}
for every $0<\gamma<1$ and $f\in L^0(X)$. 
Also, if $x\in B$, then $x\in C_i$ for large enough $i$.
Define
\[
\cB'(x)=\big\{C\in\cC: x\in C,\, C\subset B \text{ for some } B\in\cB \text{ and } \mu(C)\ge\frac12\mu(B)\big\}.
\]
Then $\cB'=\cup_{x\in X}\cB'(x)$ is clearly a Busemann--Feller basis. By \eqref{lim},
\[
\M^\gamma_\cB f(x)\le \M^\gamma_{\cB'}f(x)
\]
for every $0<\gamma<1$, $f\in L^0(X)$ and $x\in X$. 
If $C\in\cB'$, there exists $B\in\cB$ such that $C\subset B$ and $\mu(B)\le 2 \mu(C)$. By Lemma \ref{median lemma}(3), $m^\gamma_{|f|}(C)\le m^{\gamma/2}_{|f|}(B)$ for every $0<\gamma<1$ and $f\in L^0(X)$. It follows that
\[
\M^\gamma_{\cB'} f(x)\le \M^{\gamma/2}_{\cB}f(x)
\]
for every $0<\gamma<1$, $f\in L^0(X)$ and $x\in X$.
\end{proof}

\begin{theorem}
Assume that $X$ is separable and let $\cB$  be a Busemann--Feller basis on $X$.  Let $0<\gamma<1$ and $0<p<\infty$. Assume that, for every $f\in L^p(X)$, $\M^{\gamma/4}_\cB f(x)<\infty$ for almost every $x\in X$.
Then, for every $\lambda>0$, 
\[
\mu(\{x\in X: \M^\gamma_\cB f_i(x)>\lambda\})\to 0
\] as $\|f_i\|_{L^p(X)}\to 0$.
\end{theorem}
\begin{proof} 
The set $L^0(X)$ of measurable almost everywhere finite functions equipped with a functional
\[
\|f\|_{L^0(X)}=\inf_{\lambda>0}\big\{\lambda+\mu(\{x\in X:|u(x)|>\lambda\})\big\}
\]
is a quasi-normed space in the sense of Yosida \cite[p.\,30]{Y}. In particular,
\begin{equation}
\|f+g\|_{L^0(X)}\le\|g\|_{L^0(X)}+\|g\|_{L^0(X)}
\end{equation}
for every $f,g\in L^0(X)$ and
\begin{equation}\label{L^0 prop}
\|a_m f\|_{L^0(X)}\to 0 \ \text{ whenever } \ f\in L^0(X) \text{ and } a_m\to 0.
\end{equation}
It is easy to see that \[\lim_{i\to\infty}\mu(\{x\in X:|f_i(x)-f(x)|>\lambda\})=0,\] for every $\lambda>0$, if and only if $\lim_{i\to \infty}\|f_i-f\|_{L^0(X)}=0$. 

By Lemma \ref{countable}, there exists a countable Busemann--Feller basis $\cB'=\{B_j:j=1,2,\dots\}$ such that
\begin{equation}\label{B B' 1}
%\[
\M^\gamma_\cB f(x)\le \M^\gamma_{\cB'}f(x)
%\]
\end{equation}
and
\begin{equation}\label{B B' 2}
%\[
\M^{\gamma/2}_{\cB'}f(x)\le \M^{\gamma/4}_\cB f(x)
%\]
\end{equation}
for every $f\in L^0(X)$ and $x\in X$.
For $k=1,2,\dots$, denote 
\[
\M^{\gamma/2}_k f(x)=\max\big\{m^{\gamma/2}_{|f|}(B_j): B_j\ni x,\, 1\le j\le k\big\}.
\]
By Lemma \ref{med lemma 2}, operators $\M^{\gamma/2}_k$ are continuous from $L^p(X)$, $p\ge 0$, to $L^0(X)$. Let $\eps>0$. For $m\in\n$, denote 
\[
F_m=\{f\in L^p(X): \sup_k\|m^{-1}M^{\gamma/2}_kf\|_{L^0(X)}\le \eps\}.
\] 
By continuity, sets $F_m$ are closed. Since, by assumption and \eqref{B B' 2}, $\M^{\gamma/2}_{\cB'}f\in L^0(X)$, \eqref{L^0 prop} implies that
\[ 
 \sup_k\|m^{-1}M^{\gamma/2}_kf\|_{L^0(X)}
 \le \|m^{-1}M^{\gamma/2}_{\cB'}f\|_{L^0(X)}\to 0
\]
as $m\to \infty$. It follows that $L^p(X)=\cup_{m=1}^\infty F_m$. By the Baire--Hausdorff theorem (\cite[p.\,11]{Y}), one of the sets $F_m$ must have non-empty interior. Thus, there is $m_0\in\n$, $f_0\in F_{m_0}$ and $\delta>0$ such that $\|f\|_{L^p(X)}<\delta$ implies that $f_0+f\in F_{m_0}$. Hence, if $\|f\|_{L^p(X)}<\delta$, we have 
\[
\|\M^\gamma_k m_0^{-1}f\|_{L^0(X)}%=\|\M^\gamma_k n_0^{-1}(u_0+u-u_0)\|_{L^0}
 \le \|m_0^{-1}\M^{\gamma/2}_k(f_0+f)\|_{L^0(X)}+\|m_0^{-1}\M^{\gamma/2}_k f_0\|_{L^0(X)}\le 2\eps
\]
for every $k\in\n$. It follows that $\sup_k\|\M^\gamma_kf_i\|_{L^0(X)}\to 0$ as $\|f_i\|_{L^p(X)}\to 0$. This and \eqref{B B' 1} imply that,
for every $\lambda>0$,
\[
\begin{split}
\mu(\{x\in X:\M^\gamma_{\cB} f_i(x)>\lambda\})&\le
\mu(\{x\in X:\M^\gamma_{\cB'} f_i(x)>\lambda\})\\
&=\sup_k\mu(\{x\in X:\M^\gamma_kf_i(x)>\lambda\})\to 0
\end{split}
\]
as $\|f_i\|_{L^p(X)}\to 0$.
\end{proof}

To prove implication $(5)\implies (1)$ of Theorem \ref{thm1 in rn},  we employ yet another characterization of a density basis, see \cite[Theorem 1.1]{Guz1}.

\begin{theorem}\label{Guz1 thm 1.1}
Let $\cB$ be a Busemann--Feller basis. Then the following
are equivalent.
\begin{itemize}
\item[(1)] $\cB$ is a density basis. 
\item[(2)] For every $0<\gamma<1$, for every nonincreasing sequence $(r_k)$ such that $r_k\to 0$ as $k\to\infty$, and for every
nonincreasing sequence $(A_k)$ of bounded measurable sets such that $|A_k|\to 0$ as $k\to\infty$, we have
\begin{equation}\label{Guz thm 1.1 b}
|\{x\in\rn:\M_{\cB,r_k}\chi_{A_k}(x)>\gamma\}|\to 0
\end{equation}
as $k\to\infty$.
\end{itemize}
\end{theorem}
\emph{Proof of implication $(5)\implies (1)$ of Theorem \ref{thm1 in rn}}.
By Theorem \ref{Guz1 thm 1.1}, it suffices to show that
\[
|\{x\in\rn:\M_{\cB}\chi_{A_k}(x)>\gamma\}|\to 0
\]
as $k\to\infty$, whenever $0<\gamma<1$ and $(A_k)$ is a sequence of bounded measurable sets such that $|A_k|\to 0$ as $k\to\infty$.
Let $(A_k)$ be such a sequence. Then, $\|\chi_{A_k}\|_{L^p(X)}\to 0$ as $k\to\infty$. By applying Lemma \ref{Mgamma M lemma} with $f=\chi_{A_k}$ and $\lambda=1/2$, we have
\[
|\{x\in\rn:\M_{\cB}\chi_{A_k}(x)>\gamma\}|= |\{x\in\rn:\M^\gamma_\cB \chi_{A_k}(x)>1/2\}|,
\]
which, by assumption, tends to zero as $k\to\infty$. \hfill $\Box$

\section{Haj\l asz  spaces, Haj\l asz--Besov and Haj\l asz--Triebel--Lizorkin spaces}
%\subsection{Haj\l asz  spaces, Haj\l asz--Besov and Haj\l asz--Triebel--Lizorkin spaces}
Among several definitions for Besov and Triebel--Lizorkin spaces in metric measure spaces, we use the one introduced in \cite{KYZ}. This definition is motivated by the Haj\l asz--Sobolev spaces $M^{s,p}(X)$, defined for $s=1$, $p\ge1$ in \cite{H} and for fractional scales in \cite{Y}. 
We begin with the definition of a gradient on metric measure spaces.

\begin{definition}
Let $0<s<\infty$. 
A measurable function $g:X\to[0,\infty]$ is an  {\em $s$-gradient} of a function $u\in L^0(X)$ if there exists a set
$E\subset X$ with $\mu(E)=0$ such that for all $x,y\in X\setminus E$,
\begin{equation}\label{eq: gradient}
|u(x)-u(y)|\le d(x,y)^s(g(x)+g(y)).
\end{equation}
The collection of all $s$-gradients of $u$ is denoted by $\cD^s(u)$.
\end{definition}

Then we define Sobolev spaces on metric measure spaces.

\begin{definition}
Let $0<p\le\infty$. The {\em homogeneous Haj\l asz space} $\dot{M}^{s,p}(X)$ consists of measurable functions $u$ for which
\[
\|u\|_{\dot M^{s,p}(X)}=\inf_{g\in\mathcal{D}^s(u)}\|g\|_{L^p(X)}
\]
is finite.
The {\em Haj\l asz space} $M^{s,p}(X)$ is $\dot M^{s,p}(X)\cap L^p(X)$  equipped with the norm
\[
\|u\|_{M^{s,p}(X)}=\|u\|_{L^p(X)}+\|u\|_{\dot M^{s,p}(X)}.
\]
\end{definition}

Recall that for $p>1$, $M^{1,p}(\rn)=W^{1,p}(\rn)$, see \cite{H}, whereas for $n/(n+1)<p\le 1$, $M^{1,p}(\rn)$ coincides with the Hardy--Sobolev space $H^{1,p}(\rn)$ by \cite[Theorem 1]{KS}.

%\begin{remark}
%If $0<s\le 1$, then $d^s$ is a metric, $(X,d^s,\mu)$ is a doubling metric measure space and 
%$M^{s,p}(X,d,\mu)=M^{1,p}(X,d^s,\mu)$.
%\end{remark}

For the definition of Haj\l asz--Triebel--Lizorkin and Haj\l asz--Besov spaces, we need a concept of a fractional gradient, which consists of a sequence of gradient functions.

\begin{definition}
Let $0<s<\infty$.
A sequence of nonnegative measurable functions $(g_k)_{k\in\z}$ is a  {\em fractional $s$-gradient} of a function
$u\in L^0(X)$, if there exists a set $E\subset X$ with $\mu(E)=0$ such that
\begin{equation}\label{frac grad}
|u(x)-u(y)|\le d(x,y)^s(g_k(x)+g_k(y))
\end{equation}
for all $k\in\z$ and all $x,y\in X\setminus E$ satisfying $2^{-k-1}\le d(x,y)<2^{-k}$.
The collection of all fractional $s$-gradients of $u$ is denoted by $\D^s(u)$.
\end{definition}

\begin{comment}

Next three lemmas for fractional $s$-gradients follow easily from the definition. 
The corresponding results for $1$-gradients have been proved in \cite[Lemmas 2.4 and 2.5]{KiMa} and \cite[Lemma 2.6]{KL}. 
We leave the proofs for the reader.

\begin{lemma}\label{max lemma}
Let $u,v\in L^0(X)$, $(g_k)_{k\in Z}\in\D^s(u)$ and $(h_k)_{k\in Z}\in\D^s(v)$. 
Then sequence $(\max\{g_k,h_k\})_{k\in \z}$ is a fractional $s$-gradient of functions $\max\{u,v\}$ and $\min\{u,v\}$.
\end{lemma}

\begin{lemma} Let $u\in L^0(X)$. 
If there are functions $u_i\in L^0(X)$, $i\in\n$, and fractional $s$-gradients $(g_{i,k})_{k\in\z}\in \D^s(u_i)$ such that $\lim_{i\to\infty}u_i(x)=u(x)$ for almost every $x\in X$ and, for every $k$, $\lim_{i\to\infty}g_{i,k}(x)=g_{k}(x)$ for almost every $x\in X$, 
then $(g_k)_{k\in\z}\in\D^s(u)$.
\end{lemma}

\begin{lemma}\label{sup lemma} 
Let $u_i\in L^0(X)$ and $(g_{i,k})_{k\in\z}\in \D^s(u_i)$, $i\in\n$. 
Let $u=\sup_{i\in \n} u_i$ and $(g_k)_{k\in\z}=(\sup_{i\in\n}g_{i,k})_{k\in\z}$.
If $u\in L^0(X)$, then $(g_k)_{k\in\z}\in \D^s(u)$.
\end{lemma}

\end{comment}

For $0<p,q\le\infty$ and a sequence $(f_k)_{k\in\z}$ of measurable
functions, we define
\[
\big\|(f_k)_{k\in\z}\big\|_{L^p(X,\,l^q)}
=\big\|\|(f_k)_{k\in\z}\|_{l^q}\big\|_{L^p(X)}
\]
and
\[
\big\|(f_k)_{k\in\z}\big\|_{l^q(L^p(X))}
=\big\|\big(\|f_k\|_{L^p(X)}\big)_{k\in\z}\big\|_{l^q},
\]
where
\[
\big\|(f_k)_{k\in\z}\big\|_{l^{q}}
=
\begin{cases}
\big(\sum_{k\in\z}|f_{k}|^{q}\big)^{1/q},&\quad\text{if }0<q<\infty, \\
\;\sup_{k\in\z}|f_{k}|,&\quad\text{if }q=\infty.
\end{cases}
\]

Next we recall the definition of Haj\l asz--Triebel--Lizorkin and Haj\l asz--Besov spaces on metric measure spaces.

\begin{definition}
Let $0<s<\infty$ and $0<p,q\le\infty$.
\begin{itemize}
\item[(1)] The  {\em homogeneous Haj\l asz--Triebel--Lizorkin space} $\dot M_{p,q}^s(X)$ consists of functions
$u\in L^0(X)$, for which the (semi)norm
\[
\|u\|_{\dot M_{p,q}^s(X)}
=\inf_{(g_k)\in\D^s(u)}\|(g_k)\|_{L^p(X,\,l^q)}
\]
is finite. 
The  {\em Haj\l asz--Triebel--Lizorkin space} $M_{p,q}^s(X)$ is $\dot M_{p,q}^s(X)\cap L^p(X)$
equipped with the norm
\[
\|u\|_{M_{p,q}^s(X)}=\|u\|_{L^p(X)}+\|u\|_{\dot M_{p,q}^s(X)}.
\]
Notice that $M^s_{p,\infty}(X)=M^{s,p}(X)$, see \cite[Prop.\ 2.1]{KYZ}.
\item[(2)] Similarly, the  {\em homogeneous Haj\l asz--Besov space} $\dot N_{p,q}^s(X)$ consists of functions
$u\in L^0(X)$, for which
\[
\|u\|_{\dot N_{p,q}^s(X)}=\inf_{(g_k)\in\D^s(u)}\|(g_k)\|_{l^q(L^p(X))}
\]
is finite, and the {\em Haj\l asz--Besov space} $N_{p,q}^s(X)$ is $\dot N_{p,q}^s(X)\cap L^p(X)$
equipped with the norm
\[
\|u\|_{N_{p,q}^s(X)}=\|u\|_{L^p(X)}+\|u\|_{\dot N_{p,q}^s(X)}.
\]
\end{itemize}
\end{definition}
For $0<s<1$, $0<p,q\le \infty$, the spaces $N^s_{p,q}(\rn)$ and $M^s_{p,q}(\rn)$ coincide with the classical Besov and Triebel--Lizorkin spaces defined via differences,  see \cite{GKZ}.

When $0<p<1$ or $0<q<1$, the (semi)norms defined above are actually quasi-(semi)norms, but for simplicity we call them, as well as other quasi-seminorms in this paper, just norms.

\begin{definition}
Let $0<s<\infty$, $0<p,q\le \infty$ and $\cF\in\{N^s_{p,q}(X), M^s_{p,q}(X)\}$.
The \emph{$\cF$-capacity} of a set $E\subset X$ is
\[
C_{\cF}(E)=\inf\big\{\|u\|_{\cF}^p: u\in\mathcal A_{\cF}(E)\big\},
\]
where 
\[
\mathcal A_{\cF}(E)=\big\{u\in \cF: u\ge 1 \text{ on a neighbourhood of } E\big\}
\] 
is a set of admissible functions for the capacity.
We say that a property holds \emph{$\cF$-quasieverywhere} if it holds outside a set of $\cF$-capacity zero.
\end{definition}

\begin{remark} 
It is easy to see that
\[
C_{\cF}(E)=\inf\big\{\|u\|_{\cF}^p: u\in\mathcal A_{\cF}'(E)\big\},
\]
where
$\mathcal A_{\cF}'(E)=\{u\in \cA_\cF(E): 0\le u\le 1\}$. 
\end{remark}

\begin{remark}\label{cap remark 2} 
The $\cF$-capacity is an outer capacity, which means that
\[
C_{\cF}(E)=\inf\big\{C_{\cF}(U): U\supset E,\ U \text{ is open}\big\}.
\]
\end{remark}

%Throughout the rest of the section, we assume that $\mu$ %is doubling. Also, we assume that $\cF\in\%{N^s_{p,q},M^s_{p,q}\}$,
%where $0<s<1$, $0<p,q<\infty$, or 
% $\cF=M^s_{p,\infty}=M^{s,p}$, where $0<s\le 1$ and %$0<p<\infty$. These assumptions guarantee that continuous
%functions are dense in $\cF$. 

A function $u:X\to[-\infty,\infty]$ is \emph{$\cF$-quasicontinuous}
if for every $\eps>0$, there exists a set $E\subset X$ such that $C_\cF(E)<\eps$ and the restriction of $u$ to $X\setminus E$ is continuous. 
The following lemma follows from \cite[Theorem 1.2]{HeKoTu}.

\begin{lemma}\label{}
For every $u\in\cF$, there exists an $\cF$-quasicontinuous
$u^*$ such that $u(x)=u^*(x)$ for almost every $x\in X$.
\end{lemma}

The $\cF$-quasicontinuous representative is unique in the sense that if two $\cF$-quasicon- tinuous functions coincide almost everywhere, then they actually coincide $\cF$-quasi- everywhere, see \cite{K}.
The following lemma gives a useful characterization of the capacity in terms of quasicontinuous functions.
The proof of the lemma is essentially same as the proof of \cite[Theorem 3.4]{KKM}.
For $E\subset X$, denote
\[
\mathcal{QA}_\cF(E)=\{u\in\cF: u \text{ is $\cF$-quasicontinuous and } u\ge 1 \text{ $\cF$-quasieverywhere in } E\}
\]
and
\[
\widetilde C_\cF(E)=\inf_{u\in \mathcal{QA}_\cF(E)}\|u\|_\cF^p.
\]
\begin{lemma}\label{cap lemma}
Let $0<s\le 1$, $0<p<\infty$, $0<q\le \infty$ and $\cF\in\{M^s_{p,q}(X),N^s_{p,q}(X)\}$. Then
there exists a constant $C$ such that
\[ 
\widetilde C_\cF(E)\le C_\cF(E)\le C\,\widetilde C_\cF(E)
\]
for every $E\subset X$.
\end{lemma}
\begin{proof}
To prove the first inequality, let $u\in\cA_\cF(E)$ and let $u^*$ be a quasicontinuous representative of $u$. Since $u\ge 1$ in some open set $U$ containing $E$ and $u^*=u$ almost everywhere, it follows that $\min\{0,u^*-1\}= 0$ almost everywhere in $U$. Since $\min\{0,u^*-1\}$ is quasicontinuous, the
equality actually holds quasieverywhere in $U$. Hence $u^*\ge 1$ quasieverywhere in $U$, which implies that $u^*\in\mathcal{QA}_\cF(E)$.

For the second inequality, let $v\in\mathcal{QA}_\cF(E)$. By truncation, we may assume that $0\le u\le 1$.
 Fix $0 < \eps < 1$, and choose an open set $V$ with $C_\cF(V) < \eps$ so that
$v = 1$ on $E \setminus V $ and that $v$ is continuous in $X\setminus V$.
By continuity, there is an open set
$U \subset X$ such that 
\[
\{ x \in X : v(x) > 1-\eps \}\setminus V = U \setminus V.
\]
Clearly, $E \setminus V \subset U \setminus V$.
Choose $u \in \cA_\cF(V)$ such that $\|u\|_\cF < \eps$
and that $0 \le u \le 1$.
Define $w=v/(1-\eps)+u$. Then $w\ge 1$ on $(U\setminus V)\cup V=U\cup V$ which is an open neighbourhood of $E$. Hence $w\in \cA_\cF(E)$ and so
\[
C_\cF(E)^{1/p}\le \|w\|_\cF\le C\Big(\frac{1}{1-\eps}\|v\|_\cF+\|u\|_\cF\Big)
\le C\Big(\frac{1}{1-\eps}\|v\|_\cF+\eps\Big).
\]
Since $\eps>0$ and $v\in\mathcal{QA}_\cF(E)$ are arbitrary, the desired inequality $C_\cF(E)\le C\,\widetilde C_\cF(E)$ follows.
\end{proof}

The $\cF$-capacity is not generally subadditive but, for most purposes, the following result is sufficient, see \cite[ Lemma 6.4]{HeKoTu}.

\begin{lemma}\label{r-subadd}%\textbf{[\cite[Lemma 6.4]{HeKoTu}]} 
Let $0<s<\infty$, $0<p<\infty$, $0<q\le \infty$ and 
let $\cF\in\{N^s_{p,q}(X), M^s_{p,q}(X)\}$.
Then there are constants $C\ge 1$ and $0<r\le 1$ such that
\begin{equation}\label{eq: r-subadd}
C_{\cF}\big(\bigcup_{i\in\n}E_i\big)^r\le C\sum_{i\in\n} C_{\cF}(E_i)^r
\end{equation}
for all sets $E_i\subset X$, $i\in\n$.
In fact, \eqref{eq: r-subadd} holds with $r=\min\{1,q/p\}$.
\end{lemma}

The following theorem is the main result of this section.
Its proof is a modification of the proof of \cite[Theorem 2.11]{HaKi}.

\begin{theorem}\label{limsup thm}
Let $\cB$ be a differentiation basis on a doubling metric measure space $X$. Let
$\cF\in\{N^s_{p,q}(X),M^s_{p,q}(X)\}$,
where $0<s<1$, $0<p,q<\infty$, or 
 $\cF=M^s_{p,\infty}(X)=M^{s,p}(X)$, where $0<s\le 1$ and $0<p<\infty$.
Then the following claims are equivalent. 
\begin{itemize}
\item[(1)] For every quasicontinuous $u\in\cF$, there exists a set $E$ with $C_\cF(E)=0$ such that
\[
\lim_{\cB\,\ni B\to x}m^\gamma_u(B)=u(x)
\]
for every $0<\gamma<1$ and $x\in X\setminus E$.

\item[(2)] For every $0<\gamma<1$, there exists a constant $C$ such that 
\[
C_\cF(\{x\in X:\limsup_{\cB\,\ni B\to x}m^\gamma_{|u|}(B)>\lambda\})\le C\lambda^{-p}\|u\|_\cF^p
\]
for every $\lambda>0$ and $u\in\cF$.

\item[(3)] For each $0<\gamma<1$, $\lambda>0$ and for each sequence $(u_k)$ such that $\|u_k\|_\cF\to 0$ as $k\to\infty$, we have 
\[
C_\cF(\{x\in X:\limsup_{\cB\,\ni B\to x} m^\gamma_{|u_k|}(B)>\lambda\})\to 0
\]
as $k\to\infty$.
\end{itemize}
\end{theorem}

%\begin{comment}
\begin{remark} 
If $X$ and $\cF$ are as above and $\cB=\{B(x,r):x\in X,\ r>0\}$, then by \cite[Theorem 7.7]{HeKoTu},
for every $0<\gamma<1$, there exists a constant $C$ such that 
\[
C_\cF(\{x\in X:\M^\gamma_\cB u(x)>\lambda\})\le C\lambda^{-p}\|u\|_\cF^p
\]
for every $\lambda>0$ and $u\in\cF$. It would be interesting to find out 
under what assumptions on $X$, $\cB$ and $\cF$ this type of estimate
is equivalent to conditions (1)--(3) of Theorem \ref{limsup thm}. 
\end{remark}
%\end{comment}

For the proof of Theorem \ref{limsup thm}, we need the following lemma.

\begin{lemma}\label{limsup lemma}
Let $\cB$ be a differentiation basis on a doubling metric measure space $X$. Let
$\cF\in\{N^s_{p,q}(X),M^s_{p,q}(X)\}$,
where $0<s<1$, $0<p,q<\infty$, or 
$\cF=M^s_{p,\infty}(X)=M^{s,p}(X)$, where $0<s\le 1$ and $0<p<\infty$.
Let $0<\eta<\gamma<1$. Suppose that, for every $\lambda>0$, 
\[
C_\cF(\{x\in X:\limsup_{\cB\,\ni B\to x}m^\eta_{|u_k|}(B)>\lambda\})\to 0
\]
as $\|u_k\|_\cF\to 0$. Then,
for every quasicontinuous $u\in\cF$, there exists a set $E$ with $C_\cF(E)=0$ such that
\[
\lim_{\cB\,\ni B\to x}m^\gamma_{|u-u(x)|}(B)=0
\]
for every $x\in X\setminus E$.

\end{lemma}
\begin{proof} 
%Let $u\in\cF$ and let $u^*$ be an $\cF$-quasicontinuous %representative of $u$. 
%Clearly, it suffices to show
%that
%\[
%\lim_{\cB\,\ni B\to x}m^\gamma_{u^*}(B)=u^*(x)
%\]
%$\cF$-quasieverywhere.
By \cite[Theorem 1.1]{HeKoTu}, if $0<q<\infty$, and by \cite[Proposition 4.5]{SYY}, if $q=\infty$, continuous
functions are dense in $\cF$. Let $u\in\cF$ be quasicontinuous and
let $v_k\in\cF$, $k=1,2,\dots,$ be continuous such that $\|u-v_k\|_\cF\to 0$ as $k\to\infty$. Denote $w_k=u-v_k$. 
By Lemma \ref{median lemma},
\[
\begin{split}
\limsup_{\cB\,\ni B\to x}m^{\gamma}_{|u-u(x)|}(B)&\le \limsup_{\cB\,\ni B\to x}m^{\gamma-\eta}_{|v_k-v_k(x)|}(B)+\limsup_{\cB\,\ni B\to x}m^{\eta}_{|w_k-w_k(x)|}(B)\\
&\le \limsup_{\cB\,\ni B\to x}m^{\eta}_{|w_k|}(B)+|w_k(x)|.
\end{split}
\]
Hence, by Lemma \ref{r-subadd}, 
\[
\begin{split}
&C_\cF(\{x\in X:\limsup_{\cB\,\ni B\to x}m^{\gamma}_{|u-u(x)|}(B)>\lambda\})^r
\\&\le C\big(C_\cF(\{x\in X:\limsup_{\cB\,\ni B\to x}m^\eta_{|w_k|}(B)>\lambda/2\})^r
+C_\cF(\{x\in X:|w_k(x)|>\lambda/2\})^r\big).
\end{split}
\]
By assumption, 
\[
\begin{split}
C_\cF(\{x\in X:\limsup_{\cB\,\ni B\to x}m^{\eta}_{|w_k|}(B)>\lambda/2\})\to 0
\end{split}
\]
as $k\to\infty$.
Since $|w_k|$ is quasicontinuous, Lemma \ref{cap lemma} gives
\[
\begin{split}
C_\cF(\{x\in X:|w_k(x)|>\lambda/2\}) &\le C\widetilde C_\cF(\{x\in X:|w_k(x)|>\lambda/2\})\\
&\le C2^p\lambda^{-p}\|w_k\|_\cF^p\to 0
\end{split}
\]
as $k\to\infty$.
It follows that
\[
C_\cF(\{x\in X:\limsup_{\cB\,\ni B\to x}m^{\gamma}_{|u-u(x)|}(B)>\lambda\})=0
\]
for every $\lambda>0$. Hence, by Lemma \ref{r-subadd},
\[
C_\cF(\{x\in X:\limsup_{\cB\,\ni B\to x}m^{\gamma}_{|u-u(x)|}(B)>0\})=0.
\]
\end{proof}

\emph{Proof of Theorem \ref{limsup thm}.} 
We show that $(1)$ implies $(2)$. Let $0<\gamma<1$ and $u\in\cF$. Then $|u|\in\cF$ and $\||u|\|_\cF\le\|u\|_\cF$. Let $|u|^*$ be a quasicontinuous representative of $|u|$ and let $\lambda>0$. By $(1)$ and Lemma \ref{cap lemma}, we have
\[
\begin{split}
C_\cF(\{x\in X:\limsup_{\cB\,\ni B\to x}m^\gamma_{|u|}(B)>\lambda\})
&=C_\cF(\{x\in X:\limsup_{\cB\,\ni B\to x}m^\gamma_{|u|^*}(B)>\lambda\})\\
&= C_\cF(\{x\in X: |u|^*(x)>\lambda\})\\
&\le C\widetilde C_\cF(\{x\in X: |u|^*(x)>\lambda\})\\
&\le  C\lambda^{-p}\||u|^*\|_\cF^p\\
&\le  C\lambda^{-p}\|u\|_\cF^p.
\end{split}
\]

It is clear that $(2)\implies (3)$. 

Then we show that $(3)$ implies $(1)$. Let $u\in\cF$ be quasicontinuous. By Lemma \ref{limsup lemma}, for every $k=2,3,\dots$, there exists $E_k$ such that $C_\cF(E_k)=0$ and
\[
\lim_{\cB\ni B\to x}m^{1/k}_{|u-u(x)|}(B)=0
\]
for every $x\in X\setminus E_k$. 
It follows that
\[
\lim_{\cB\ni B\to x}m^{\gamma}_{|u-u(x)|}(B)=0
\]
for every $0<\gamma<1$ and $x\in X\setminus E$, where $E=\cup_{k\ge 2}E_k$ is of $\cF$-capacity zero by Lemma \ref{r-subadd}.
By Lemma \ref{median lemma},
\[
| m^\gamma_{u}(B)-u(x)|
=| m^\gamma_{u-u(x)}(B)|
\le m^{\min\{\gamma,1-\gamma\}}_{|u-u(x)|}(B),
\]
and the claim follows. \hfill $\Box$
\bigskip

If we restrict the value of $p$ such that $\cF$-functions are
locally integrable, then  Theorem \ref{limsup thm}
has a counterpart formulated in terms of integral averages.
The proof of Theorem \ref{limsup thm for averages}, which is similar to the proof
of Theorem \ref{limsup thm}, will be omitted. 
%The case $\cF=M^{1,p}(X)$, $p>1$, was proved in \cite{HaKi}.

\begin{theorem}\label{limsup thm for averages}
Let $\cB$ be a Busemann--Feller basis on a metric measure space $X$ satisfying \eqref{doubling dim}. Let
$\cF\in\{N^s_{p,q}(X),M^s_{p,q}(X)\}$,
where $0<s<1$, $Q/(Q+s)<p<\infty$, $0<q<\infty$, or 
 $\cF=M^s_{p,\infty}(X)=M^{s,p}(X)$, where $0<s\le 1$ and $Q/(Q+s)<p<\infty$.
Then the following claims are equivalent. 
\begin{itemize}
\item[(1)] For every quasicontinuous $u\in\cF$, there exists a set $E$ with $C_\cF(E)=0$ such that
\[
\lim_{\cB\,\ni B\to x} u_B=u(x)
\]
for every $x\in X\setminus E$. 

\item[(2)] There exists a constant $C$ such that 
\[
C_\cF(\{x\in X:\limsup_{\cB\,\ni B\to x}|u|_B>\lambda\})\le C\lambda^{-p}\|u\|_\cF^p
\]
for every $u\in\cF$.

\item[(3)] For each $\lambda>0$ and for each sequence $(u_k)$ such that $\|u_k\|_\cF\to 0$ as $k\to\infty$, we have 
\[
C_\cF(\{x\in X:\limsup_{\cB\,\ni B\to x} |u_k|_B>\lambda\})\to 0
\]
as $k\to\infty$.
\end{itemize}
\end{theorem}

For the proof of our next result, Theorem \ref{M thm}, 
we need the following lemma. 
The proof given below is a modification of the proof of \cite[Theorem 4.1]{KiMa2}.
We do not know whether the lemma holds true
when $p\le 1$ or $q\le 1$.

\begin{lemma}\label{cap lemma 2}
Let $0<s\le 1$, $1<p<\infty$, $1<q\le \infty$ and $\cF\in\{M^s_{p,q}(X),N^s_{p,q}(X)\}$. Then
\[ 
C_\cF(U)= \lim_{i\to\infty}C_\cF(U_i)
\]
whenever $U_1\subset U_2\subset\cdots$ are open subsets of $X$ and $U=\cup_{i=1}^\infty U_i$.
\end{lemma}
\begin{proof}
We prove the case $\cF=M^s_{p,q}(X)$. The proof of the other case is similar.
By monotonicity, 
\[
\lim_{i\to\infty} C_{M^s_{p,q}(X)}(U_i)\le C_{M^s_{p,q}(X)}(U).
\]
To prove the opposite inequality, we may assume that
$\lim_{i\to\infty} C_{M^s_{p,q}(X)}(U_i)<\infty.$ Let $\eps>0$ and let $u_i\in \cA_{M^s_{p,q}(X)}'(U_i)$ and $\overline g_i=(g^i_k)_{k\in\z}\in \D^s(u_i)$
be such that 
\[
(\|u_i\|_{L^p(X)}+\|\overline g_i\|_{L^p(X,l^q)})^p < C_{M^s_{p,q}(X)}(U_i)+\eps.
\]
Then $(u_i)$ is bounded in $L^p(X)$ and $(\overline g_i)$ is bounded in $L^p(X,l^q)$. %Since both $L^p(X)$ and $L^p(X,l^q)$ are reflexive, 
Hence, by passing to a subsequence, we may assume that $u_i\to u$ weakly in $L^p(X)$ and $\overline g_i\to\overline g$ weakly in $L^p(X,l^q)$. Using Mazur's lemma, we obtain convex combinations $v_j=\sum_{i=j}^{n_j}\lambda_{j,i}u_i$ and $\overline h_j=\sum_{i=j}^{n_j}\lambda_{j,i}\overline g_i$ 
such that $v_j\to u$ in $L^p(X)$, $\overline h_j\to \overline g$ in $L^p(X,l^q)$
as $j\to\infty$ and $\overline h_j\in \D^s(v_j)$. Passing to a subsequence, we may also assume that $v_j\to u$ and $\overline h_j\to \overline g$ pointwise almost everywhere as $j\to\infty$.
This easily implies that $\overline g\in \D^s(u)$.
Since $u_i=1$, it follows that,
for every $x\in U$, $v_j(x)=1$ for $j$ large enough. Hence $u=1$ in $U$ and so $u\in \cA_{M^s_{p,q}(X)}'(U)$. By the weak lower semicontinuity
of norms,
\[
\begin{split}
C_{M^s_{p,q}(X)}(U)&\le (\|u\|_{L^p(X)}+\|\overline g\|_{L^p(X,l^q)})^p\\ 
&\le \liminf_{i\to\infty}\,(\|u_i\|_{L^p(X)}+\|\overline g_i\|_{L^p(X,l^q)})^p\\ &\le \lim_{i\to\infty}C_{M^s_{p,q}(X)}(U_i)+\eps
\end{split}
\]
and the claim follows by letting $\eps\to 0$.
\end{proof}

\begin{theorem}\label{M thm}
Let $\cB$ be a Busemann--Feller basis on a separable metric measure space $X$. Let
$\cF\in\{N^s_{p,q}(X),M^s_{p,q}(X)\}$,
where $0<s<1$, $1<p,q<\infty$, or 
 $\cF=M^s_{p,\infty}(X)=M^{s,p}(X)$, where $0<s\le 1$ and $1<p<\infty$. Consider the following claims.
\begin{itemize}
\item[(1)] For each $0<\gamma<1$ and $u\in\cF$, 
\[
C_\cF(\{x\in X:\M^\gamma_\cB u(x)>\lambda\})\to 0
\]
as $\lambda\to \infty$.

\item[(2)] For each $0<\gamma<1$, $\lambda>0$ and for each sequence $(u_k)$ such that $\|u_k\|_\cF\to 0$ as $k\to\infty$, we have 
\[
C_\cF(\{x\in X:\M^\gamma_\cB u_k(x)>\lambda\})\to 0
\]
as $k\to\infty$.

\item[(3)] For every quasicontinuous $u\in\cF$, there exists a set $E$ with $C_\cF(E)=0$ such that
\[
\lim_{\cB\,\ni B\to x}m^\gamma_u(B)=u(x)
\]
for every $0<\gamma<1$ and $x\in X\setminus E$. 
\end{itemize}
Then $ (1)\implies (2) \implies (3)$.
\end{theorem}

Implication $(2)\implies (3)$ of Theorem \ref{M thm} follows from Theorem \ref{limsup thm}
(because condition (2) of Theorem \ref{M thm} trivially implies condition (3)
of Theorem \ref{limsup thm}).
Hence, it suffices to prove the following lemma.

\begin{comment}
The following lemma is well-known. See e.g. \cite[Proposition 14.5]{Hei}.

\begin{lemma} 
Let $\rho$ be a quasi-metric on a set $X$. Then there exists
a metric $d$ on $X$ and a constant $0<\alpha\le 1$ such that
\[d(x,y)\le \rho(x,y)^\alpha\le 2 d(x,y)\]
for all $x,y\in X$.
\end{lemma}
\end{comment}

\begin{lemma}\label{M lemma}
Suppose that the assumptions of Theorem \ref{M thm} are in force. Let $0<\gamma<1$ and suppose that, for every $u\in \cF$,  
\[
C_{\cF}(\{x\in X: \M^{\gamma/4}_\cB u(x)>\lambda\})\to 0
\] as  $\lambda\to \infty$.
Then, for every $\lambda>0$,
\[
C_{\cF}(\{x\in X: \M^{\gamma}_\cB u_i(x)>\lambda\})\to 0
\] as $\|u_i\|_{\cF}\to 0$.
\end{lemma}
\begin{proof} 
Denote by $Y$ the set of measurable functions for which
\[
\lim_{\lambda\to\infty}C_{\cF}(\{x\in X:|u(x)|>\lambda\})=0.
\]
and equip $Y$ with a functional
\[
\|u\|_Y=\inf_{\lambda>0}\big\{\lambda+C_{\cF}(\{x\in X:|u(x)|>\lambda\})^r\big\},
\]
where $r$ is the exponent from \eqref{eq: r-subadd}.
Then $\|u\|_Y=0$ if and only if $u=0$ $\cF$-quasieverywhere.
Moreover, there exists a constant $C\ge 1$ such that
\begin{equation}
\|u+v\|_{Y}\le C(\|u\|_{Y}+\|v\|_{Y})
\end{equation}
for every $u,v\in Y$ and
\begin{equation}\label{Y prop}
\|a_n u\|_Y\to 0 \ \text{ whenever } \ u\in Y \text{ and } a_n\to 0.
\end{equation}
It is easy to see that \[\lim_{i\to\infty}C_{\cF}(\{x\in X:|u_i(x)-u(x)|>\lambda\})=0,\] for every $\lambda>0$, if and only if $\lim_{i\to \infty}\|u_i-u\|_{Y}=0$. 

Since $\widetilde d_Y(u,v)=\|u-v\|_Y$ is a quasi-metric on $Y$, there exists
a metric $d_Y$ on $Y$ and a constant $0<\alpha\le 1$ such that
\begin{equation}\label{dY}
d_Y(u,v)\le \|u-v\|_Y^\alpha\le 2d_Y(u,v)
\end{equation}
for all $u,v\in Y$, see e.g. \cite[Proposition 14.5]{Hei}.

%Similarly, there exists a metric $d_\cF$ on $\cF$ and a constant $0<\beta\le 1$ such that
%\begin{equation}\label{dF}
%d_\cF(u,v)\le \|u-v\|_\cF^\beta\le 2d_\cF(u,v)
%\end{equation}
%for all $u,v\in\cF$.

Let $\cB'=\{B_j\}$ be the countable Busemann--Feller basis given by Lemma \ref{countable}.
Denote \[\M^{\gamma/2}_k u=\max_{1\le j\le k}m^{\gamma/2}_{|u|}(B_j)\chi_{B_j}.\] Then
operators $\M^{\gamma/2}_{k}$ are continuous from 
$L^0(X)$ to $Y$ (and so from $\cF$ to $Y$). Indeed, if $u_i\to u$ in
$L^0(X)$, then by Lemma \ref{med lemma 2}, for every $j$, $m^{\gamma/2}_{|u_i|}(B_j)\to m^{\gamma/2}_{|u|}(B_j)$ as $i\to\infty$. 
Since
\[
|\M^{\gamma/2}_k u_i-\M^{\gamma/2}_k u|\le\max_{1\le j\le k}|m^{\gamma/2}_{|u_i|}(B_j)-m^{\gamma/2}_{|u|}(B_j)|\chi_{B_j},
\]
it follows that, for every $\lambda>0$,
\[
C_\cF(\{x\in X: |\M^{\gamma/2}_k u_i(x)-\M^{\gamma/2}_k u(x)|>\lambda\})=0
\]
for large enough $i$, which in turn implies that
$
d_Y(\M^{\gamma/2}_k u_i,\M^{\gamma/2}_k u)\to 0
$
as $i\to\infty$.

Let $\eps>0$. For $n\in\n$, denote 
\[
F_n=\{u\in \cF: \sup_k d_Y(n^{-1}M^{\gamma/2}_k u,0)\le \eps\}.
\] 
By continuity, sets $F_n$ are closed. Let $u\in \cF$. Since, by assumption, $\M^{\gamma/4}_{\cB}u\in Y$, Lemma \ref{countable} implies that $\M^{\gamma/2}_{\cB'}u\in Y$. Hence, by \eqref{dY} and \eqref{Y prop},
\[
\sup_k d_Y(n^{-1}M^{\gamma/2}_ku,0)\le \sup_k\|n^{-1}M^{\gamma/2}_ku\|_{Y}^\alpha\le \|n^{-1}M^{\gamma/2}_{\cB'}u\|_{Y}^\alpha\to 0
\]
as $n\to \infty$. It follows that $\cF=\cup_{n=1}^\infty F_n$. By the Baire--Hausdorff theorem (\cite[p.\,11]{Y}), one of the sets $F_n$ must have non-empty interior. Thus, there is $n_0\in\n$, $u_0\in F_{n_0}$ and $\delta>0$ such that $B_\cF(u_0,\delta)\subset F_{n_0}$. Now, if $\|u\|_{\cF}<\delta$, 
then $u+u_0\in F_{n_0}$, and so
\[
\|\M^\gamma_k n_0^{-1}u\|_{Y}%=\|\M^\gamma_k n_0^{-1}(u_0+u-u_0)\|_{L^0}
 \le C\big(\|n_0^{-1}\M^{\gamma/2}_k(u_0+u)\|_{Y}+\|n_0^{-1}\M^{\gamma/2}_k u_0\|_{Y}\big)\le C\eps
\]
for every $k\in\n$. It follows that $\sup_k\|\M^\gamma_ku_i\|_{Y}\to 0$ as $\|u_i\|_{\cF}\to 0$. This, Lemma \ref{countable} and Lemma \ref{cap lemma 2} imply that, 
for every $\lambda>0$,
\[
\begin{split}
C_\cF(\{x\in X:\M^\gamma_\cB u_i(x)>\lambda\})&\le
C_\cF(\{x\in X:\M^\gamma_{\cB'} u_i(x)>\lambda\})\\
&=\sup_k C_\cF(\{x\in X:\M^\gamma_ku_i(x)>\lambda\})\to 0
\end{split}
\]
as $\|u_i\|_{\cF}\to 0$.
\end{proof}

%\end{comment}

A similar reasoning as above gives the following
result for the usual maximal function. The proof of the theorem will be omitted.

\begin{theorem}\label{M thm 2}
Let $\cB$ be a Busemann--Feller basis on a separable metric measure space $X$. Let
$\cF\in\{N^s_{p,q}(X),M^s_{p,q}(X)\}$,
where $0<s<1$, $1<p,q<\infty$, or 
 $\cF=M^s_{p,\infty}(X)=M^{s,p}(X)$, where $0<s\le 1$ and $1<p<\infty$. Consider the following claims.
\begin{itemize}
\item[(1)] For each $u\in\cF$, 
\[
C_\cF(\{x\in X:\M_\cB u(x)>\lambda\})\to 0
\]
as $\lambda\to \infty$.

\item[(2)] For each $\lambda>0$ and for each sequence $(u_k)$ such that $\|u_k\|_\cF\to 0$ as $k\to\infty$, we have 
\[
C_\cF(\{x\in X:\M_\cB u_k(x)>\lambda\})\to 0
\]
as $k\to\infty$.

\item[(3)] For every quasicontinuous $u\in\cF$, there exists a set $E$ with $C_\cF(E)=0$ such that
\[
\lim_{\cB\,\ni B\to x}u_{B}=u(x)
\]
for every $x\in X\setminus E$. 
\end{itemize}
Then $ (1)\implies (2) \implies (3)$.
\end{theorem}

%\vspace{0.5cm}
%\noindent
%\small{\textsc{P.K.},}
%\small{\textsc{Department of Mathematics and Statistics},}
%\small{\textsc{P.O. Box 35},}
%\small{\textsc{FI-40014 University of Jyv\"askyl\"a},}
%\small{\textsc{Finland}}\\
%\footnotesize{\texttt{pekka.j.koskela@jyu.fi}}

%\vspace{0.5cm}
%\noindent
%\small{\textsc{T.H.},}
%\small{\textsc{Department of Mathematics},}
%\small{\textsc{P.O. Box 11100},}
%\small{\textsc{FI-00076 Aalto University},}
%\small{\textsc{Finland}}\\
%\footnotesize{\texttt{toni.heikkinen@aalto.fi}}

%\vspace{0.5cm}
%\noindent
%\small{\textsc{T.H.},}
%\small{\textsc{Department of Mathematics and Statistics},}
%\small{\textsc{P.O. Box 35},}
%\small{\textsc{FI-40014 University of Jyv\"askyl\"a},}
%\small{\textsc{Finland}}\\
%\footnotesize{\texttt{toni.heikkinen@aalto.fi}}
%\todo{Toni, kumpi osoite?}

%\vspace{0.3cm}
%\noindent
%\small{\textsc{H.T.},}
%\small{\textsc{Department of Mathematics and Statistics},}
%\small{\textsc{P.O. Box 35},}
%\small{\textsc{FI-40014 University of Jyv\"askyl\"a},}
%\small{\textsc{Finland}}\\
%\footnotesize{\texttt{heli.m.tuominen@jyu.fi}}

\end{document}